\documentclass[a4paper,11pt]{article}
\usepackage[utf8]{inputenc}
\usepackage{amsmath, amsthm, amsfonts, amssymb}
\usepackage{mathtools}
\usepackage{combelow}
\usepackage[pdftex, bookmarksnumbered=true]{hyperref}
\title{Twisted and Non-Twisted Deformed Virasoro Algebra \\via Vertex Operators of $\qsl$}
\author{Mikhail Bershtein \and Roman Gonin}

\theoremstyle{definition}
\newtheorem{definition}{Definition}
\theoremstyle{plain}
\newtheorem{theorem}{Theorem}
\newtheorem{proposition}{Proposition}
\newtheorem{corollary}[proposition]{Corollary}
\newtheorem{lemma}[proposition]{Lemma}
\theoremstyle{remark}
\newtheorem{remark}{Remark}
\newtheorem{conjecture-prop}{Conjecture-Proposition}
\usepackage{color}

\numberwithin{proposition}{section}
\numberwithin{definition}{section}

\newcommand{\fp}{{\beta}}
\newcommand{\pp}{{\alpha}_{\psi}}
\newcommand{\ff}{{\alpha}_{\phi}}
\DeclareMathOperator{\res}{res}
\DeclareMathOperator{\Hom}{Hom}
\newcommand{\dood}{\delta_{\text{odd}}}
\newcommand{\asl}{\widehat{\mathfrak{sl}}_2}
\newcommand{\qsl}{U_q(\widehat{\mathfrak{sl}}_2)}
\newcommand{\vac}{| \varnothing \rangle }
\numberwithin{equation}{section}
\newcommand{\Vir}{\mathsf{Vir}_{q_1,q_2}}
\newcommand{\Virt}{\mathsf{Vir}^{\text{tw}}_{q_1,q_2}}
\newcommand{\RepV}{\mathcal{F}_{u}^{[j]}}
\newcommand{\RepVt}{\mathcal{F}^{[i]}}
\newcommand{\jvec}{|j \rangle}
\newcommand{\tVerm}{M^{\text{tw}}}
\newcommand{\degp}{\deg_{\text{pr}}}
\begin{document}
	
	\maketitle
	\begin{abstract}
		The work is devoted to a probably new connection between deformed Virasoro algebra and quantum $\asl$. We give an explicit realization of Virasoro current via vertex operators of level 1 integrable representation of $\asl$. The same is done for a twisted version of deformed Virasoro algebra.   
	\end{abstract}
	
	\section{Introduction}
	The work is devoted to a probably new connection between deformed Virasoro algebra and quantum $\asl$. More specifically, we use integrable representations  $V(\Lambda_0)$ and $V(\Lambda_1)$ of quantum $\asl$ on level 1, they possess realization in terms of Heisenberg algebra. Also there are vertex operators
	\begin{align} \label{eq: intertwiner intro}
	\Phi(z)& \colon V(\Lambda_i) \rightarrow V(\Lambda_{1-i}) \otimes V_{z}, & \Psi(z) &\colon V(\Lambda_i) \rightarrow V_{z}  \otimes V(\Lambda_{1-i}).
	\end{align}
	for evaluation representation $V_z$ of $\asl$. The main result is realization of deformed Virasoro algebra in terms of these vertex operators see Theorem \ref{th: Vir non twist} and Theorem \ref{th: Vir twist}. Let us remark that deformed Virasoro algebra depend on parameters $q_1$, $q_2$, $q_3$ such that $q_1 q_2 q_3 =1$. It turns out, that deformed Virasoro algebra is connected with $\qsl$ for $q =q_3^{1/2}$.
	
	To be more precise, in Theorem \ref{th: Vir twist} we have a realization of twisted deformed Virasoro algebra. This algebra was defined in \cite[(37)--(38)]{Sh} but its bosonization was unknown.
	
	In Theorem \ref{th: Vir non twist} we have constructed realization of ordinary (non-twisted)  deformed Virasoro algebra defined in \cite{SKAO}. A bosonization of this algebra is known since \cite{SKAO}, but our bosonization is a different one. The bosonization from \cite{SKAO} can be also realized by the same formula as in Theorem \ref{th: Vir non twist} but using another vertex operators \cite{Ding1997} defined by \eqref{eq: intertwiner intro} with respect to Drinfeld coproduct.
	
	The existence of two realizations (our Theorems \ref{th: Vir non twist} and \ref{th: Vir twist}) is similar to existence of two choices of twist in XXZ model, see e.g. \cite[eq. 4.3]{Maillet2018}.
	
	\paragraph{Further development} Deformed Virasoro algebra is a particular case of $W_{q_1,q_2}(\mathfrak{sl}_n)$ for $n=2$. Twisted deformed Virasoro is a particular case of twisted $W$-algebras $W_{q_1,q_2}(\mathfrak{sl}_n, n_{tw})$ for $n=2$ and $n_{tw}=1$. Generally $n_{tw}$ is a parameter of twist, i.e. for $n_{tw}=0$ we obtain non-twisted $W$-algebra. We expect that one can construct bosonization of $W_{q_1,q_2}(\mathfrak{sl}_n, n_{tw})$ algebra from vertex operators of quantum $\widehat{\mathfrak{sl}}_n$ on the level 1 (see \cite{Koyama1993}).
	
	Also we expect that tensor product of $W_{q_1,q_2}(\mathfrak{sl}_n, n_{tw})$ with Heisenberg algebra $\mathsf{Heis}$ are certain quotients of toroidal algebra $U_{q_1,q_2,q_3}(\ddot{\mathfrak{gl}}_1)$; for non-twisted case such relation is known \cite{FHSSY}, \cite{FFJMM}, \cite{N16}. Hence a representation of $W_{q_1,q_2}(\mathfrak{sl}_n, n_{tw}) \otimes \mathsf{Heis}$ becomes a representation of  $U_{q_1,q_2,q_3}(\ddot{\mathfrak{gl}}_1)$ automatically. We expect that for $\gcd(n,n_{tw})=1$, the discussed above bosonized representation of $W_{q_1,q_2}(\mathfrak{sl}_n, n_{tw})$ will lead to Fock modules of $U_{q_1,q_2,q_3}(\ddot{\mathfrak{gl}}_1)$ with slope $n_{tw}/n$. For the case $q_3=1$ all this was done in \cite{BG}. Let us remark $W_{q_1,q_2}(\mathfrak{sl}_n, n_{tw}) \otimes \mathsf{Heis}$ acts on integrable level 1 representations of $U_q (\widehat{\mathfrak{gl}}_n) = U_q (\widehat{\mathfrak{sl}}_n) \otimes \mathsf{Heis}$ if this holds without $\mathsf{Heis}$ factors.
	
	One of our motivations for this project comes from \cite{GN15}. It was conjectured in \emph{loc. cit.} that there is an action (with certain properties) of $U_{q}\widehat{\mathfrak{gl}}_n$ on the Fock module of toroidal algebra $U_{q_1,q_2,q_3}(\ddot{\mathfrak{gl}}_1)$ with slope $n'/n$. As it was explained above, we also expect $U_{q}\widehat{\mathfrak{gl}}_n$ acts on the Fock module of toroidal algebra $U_{q_1,q_2,q_3}(\ddot{\mathfrak{gl}}_1)$. So we hope that both actions exist and coincide.
	
	\paragraph{Our methods.} The main technical tool of our paper is R-matrix relations (Theorems \ref{Th: R-matrix for Phi} and \ref{Th: R-matrix for Psi}). One can find these relations without delta-function term in \cite{JM}. In \emph{loc. cit.} the parameters of vertex operator are numbers, but in our paper the parameters are formal variables. Therefore our formulas are close to formulas in \emph{loc. cit.}, but have a different meaning and probably are new. 
	
	Technically, we write down formulas \eqref{eq: realization Vir non twist} and \eqref{eq: realization Vir twist} for current of deformed Virasoro algebra $T(z)$ via vertex operators $\Phi(z)$ and $\Psi^*(z)$, and then we check relations of deformed Virasoro algebra using interchanging relations for vertex operators. Delta-function term on RHS of deformed Virasoro relation appears from the delta-function term in the R-matrix relation.
	\paragraph{Plan of the paper.} The paper is organized as follows
	
	In Section \ref{sec: boson of sl} we recall bosonization of $\qsl$ and its vertex operators following \cite{JM}.
	
	In Section \ref{Section: JM} we study relations for vertex operators: interchanging relations (in particular R-matrix relations), and `special point relations'.
	
	In Section \ref{sec: Vir} we construct realizations of twisted and non-twisted Virasoro algebra via vertex operators of $\qsl$. A connection of obtained representations with Verma modules is studied.
	
	\paragraph{Acknowledgments} We are grateful to B.~Feigin,  E.~Gorsky, M.~Jimbo, M.~Lashkevich, A.~Negu\cb{t}, Y.~Pugai, J.~Shiraishi for interest to our work and discussions. The work is supported by the Russian Foundation of Basic Research under grant mol\_a\_ved  18-31-2006. R.G. was also supported by the ``Young Russian Mathematics" award.
	
	\section{Bosonization of $\qsl$ and its vertex operators} \label{sec: boson of sl}
	In this section we will recall the bosonization of the level 1 representations of $\qsl$ and its vertex operators. All this can be found in \cite[Chapters 5,6]{JM}. Our notation almost coincide with \cite{JM}; however there are differences in normalization of vertex operators.
	\paragraph{Fock modules.} Algebra $\qsl$ is generated by $x_k^{\pm}$, $a_l$ for $k \in \mathbb{Z}$, $l \in \mathbb{Z}_{\neq 0}$, $K^{\pm 1}$ and central elements $\gamma^{\pm 1/2}$. These elements are called \emph{Drinfeld generators}. The relations are \cite[(5.3)--(5.7)]{JM}, although let us recall 
	\begin{equation}
	[a_k, a_l]= \delta_{k+l,0}\frac{[2k]}{k} \frac{\gamma^k - \gamma^{-k}}{q - q^{-1}},
	\end{equation}
	here $[n]=(q^n-q^{-n})/(q-q^{-1})$.
	
	Denote by $\Lambda_0$, $\Lambda_1$ fundamental weights of $\widehat{\mathfrak{sl}}_2$ and by $\alpha$ root of $\mathfrak{sl}_2 \subset \widehat{\mathfrak{sl}}_2$
	The algebra $\qsl$ admits two basic representations $V(\Lambda_0)$ and $V(\Lambda_1)$. As vector spaces
	\begin{equation}
	V(\Lambda_i) = \mathbb{C}[a_{-1}, a_{-2}, \dots ] \otimes \left( \oplus_{n} \mathbb{C} e^{\Lambda_i + n \alpha} \right).
	\end{equation}
	As representations of Heisenberg subalgebra $a_k$, these modules are infinite sums of Fock modules 
	\begin{equation} \label{notatio: Vj}
	V_j = \mathbb{C}[a_{-1}, a_{-2}, \dots ] \otimes \mathbb{C} e^{\Lambda_i + \lfloor \frac{j}{2} \rfloor \alpha} \quad \text{for $i \equiv j \bmod 2$.}
	\end{equation}
	Let us define operators $e^{\pm \alpha}$ and $\partial$ as follows
	\begin{equation}
	e^{\pm \alpha} \left( f \otimes e^{\beta} \right) = f \otimes e^{\beta \pm \alpha}, \quad \quad \partial \left( f \otimes e^{\beta} \right) = (\alpha, \beta) f \otimes e^{\beta}.
	\end{equation}
	The action of other $\qsl$ generators is given by
	\begin{align}
	&K =q^{\partial}, \quad \gamma = q, 
	\\
	&X^{\pm}(z)= \exp\left( \pm \sum_{n=1}^{\infty} \frac{a_{-n}}{[n]} q^{\mp n/2} z^n \right) \exp \left( \mp \sum_{n=1}^{\infty} \frac{a_n}{[n]} q^{\mp n/2} z^{-n}\right) e^{\pm \alpha} z^{\pm \partial},
	\end{align}
	here $X^{\pm}(z) = \sum x_k^{\pm} z^{-k-1}$. Obtained representations are irreducible highest weight representations with highest vectors $|\Lambda_i \rangle = 1 \otimes e^{\Lambda_i} \in V(\Lambda_i)$.
	\paragraph{Vertex operators.} Vertex operator of $\qsl$ are certain formal power series of operators 
	\begin{align}
	\Phi^{(1-i, i)}_{\pm}(z) &\colon V(\Lambda_i) \rightarrow V(\Lambda_{1-i}), & \Psi^{(1-i, i)}_{\pm}(z) &\colon V(\Lambda_i) \rightarrow V(\Lambda_{1-i}).
	\end{align}
	Below we will abbreviate $\Phi_{\pm}(z) = \Phi^{(1-i, i)}_{\pm}(z)$, if a statement holds for both $i=0,1$.
	
	A conceptual definition of these operators via certain intertwiner relations is given in \cite[Chapter 6]{JM}. For us it is more convenient to give an \emph{ad hoc} definition
	\begin{align}
	\Phi_-(z) =& \exp \left( \sum_{n=1}^{\infty} \frac{a_{-n}}{[2n]} q^{7n/2} z^n \right) \exp \left(- \sum_{n=1}^{\infty} \frac{a_{n}}{[2n]} q^{-5n/2} z^{-n} \right) e^{\alpha/2} (-q^3 z)^{\partial/2}, \label{eq: Phi- boson} \\
	\Psi_+(z) =& \exp \left( - \sum_{n=1}^{\infty} \frac{a_{-n}}{[2n]} q^{n/2} z^n \right) \exp \left( \sum_{n=1}^{\infty} \frac{a_{n}}{[2n]} q^{-3n/2} z^{-n} \right) e^{-\alpha/2} (-q z)^{-\partial/2}.
	\label{eq: Psi+ boson}
	\end{align}
	Operator $\Phi_+(z)$, $\Psi_-(z)$ are determined by each of these formulas
	\begin{align}
	\Phi_+ (z) &= [\Phi_- (z) , x_0^- ]_{q},  & q^2 z \Phi_+ (z) =& [ \Phi_- (z), x_1^- ]_{q^{-1}}, \label{eq: Phi min x_0}\\
	\Psi_-(z) &= [ \Psi_+(z),  x_0^+]_{q},  & (q^2 z)^{-1} \Psi_-(z) =& [\Psi_+(z), x_{-1}^+]_{q^{-1}}. \label{eq: Psi min def}
	\end{align}
	here we use following notation $[A,B]_{p}= AB-pBA$.
	
	We will also need the dual operator $\Psi_{\varepsilon}^*(z) = \Psi_{-\varepsilon}(q^2 z)$. Then
	\begin{equation}
	\Psi^*_-(z) = \exp \left( - \sum \frac{a_{-n}}{[2n]} q^{5n/2} z^n \right) \exp \left( \sum \frac{a_{n}}{[2n]} q^{-7n/2} z^{-n} \right) e^{-\alpha/2} (-q^3 z)^{-\partial/2}.
	\label{eq: Psi-* boson}
	\end{equation}

	Denote
	\begin{align} \label{eqdef: alpha and beta}
	\ff(x) =& \frac{(q^4 x; q^4)_{\infty}}{(q^2 x; q^4)_{\infty}}, &
	\pp(x) =& \frac{(q^2 x; q^4)_{\infty}}{(x; q^4)_{\infty}},&
	\fp(x) =& \frac{(q x; q^4)_{\infty}}{(q^3 x; q^4)_{\infty}}.
	\end{align}
	It is straightforward to check that
	\begin{align}
	(-q^3 z)^{-1/2} \ff(w/z) \Phi_-(z) \Phi_-(w) =& \ :\!\Phi_-(z) \Phi_-(w)\!:,  \label{eq: Phi Phi norm ord}\\
	(-qz)^{-1/2} \pp(w/z) \Psi_+(z) \Psi_+(w) =& \ :\!\Psi_+(z) \Psi_+(w)\!:, \label{eq: Psi Psi norm ord}\\
	(-q^3 z)^{1/2} \fp(w/z) \Phi_-(z) \Psi_-^*(w) =& \ :\!\Phi_-(z) \Psi_-^*(w)\!:, \label{eq: Phi Psi norm ord}\\
	(-q^3 w)^{1/2} \fp(z/w) \Psi_-^*(w) \Phi_-(z) =& \ :\!\Phi_-(z) \Psi_-^*(w)\!:\!. \label{eq: Psi Phi norm ord}
	\end{align}
	here $: \ldots :$ stands for normal ordering in terms of Heisenberg algebra. Then 
	\begin{align}
	z^{-1/2} \ff(w/z) \Phi_-(z) \Phi_-(w) =& w^{-1/2} \ff(z/w) \Phi_-(w) \Phi_-(z), \label{eq: interchange phi- phi-} \\
	z^{-1/2} \pp(w/z) \Psi^*_-(z) \Psi^*_-(w) =& w^{-1/2} \pp(z/w) \Psi^*_-(w) \Psi^*_-(z),\label{eq: interchange psi- psi-} \\
	z^{1/2} \fp(w/z) \Phi_-(z) \Psi_-^*(w) =& w^{1/2} \fp(z/w) \Psi_-^*(w) \Phi_-(z). \label{eq: interchange psi- phi-}
	\end{align}
	\paragraph{$\pi$-involution} Recall, that $\qsl$ is also generated by $e_i$, $f_i$, $t_i$ for $i=0,1$. These generators are called \emph{Chevalley generators}. The connection with Drinfeld generators as follows
	\begin{align}
	t_1=&K, & x_0^+=&e_1, & x_0^-=&f_1, \\
	t_0 =& \gamma K^{-1}, & x_1^- =& e_0 t_1, & x_{-1}^+ =& t_1^{-1} f_0.
	\end{align}
	Let us consider an exterior automorphism $\pi$ of $\qsl$ given by $\pi(e_i)=e_{1-i}$,  $\pi(f_i)=f_{1-i}$. Then $\pi$ acts on the Drinfeld generators as follows
	\begin{align}
	&\pi(K)= \gamma K^{-1}, & &\pi(x_0^+) =  x_1^- K^{-1}, & &\pi(x_0^-) =  K  x_{-1}^+, \\
	&   &  &\pi(x_1^-)=  \gamma x_0^+ K^{-1},& &\pi(x_{-1}^+) = \gamma^{-1} K x_0^{-}.
	\end{align}
	\begin{proposition}
		
		There exist an involution $\tilde{\pi}$ interchanging $V(\Lambda_0)$ and $V(\Lambda_1)$, such that $\tilde{\pi} X \tilde{\pi} = \pi(X)$ for any $X \in \qsl$.
	\end{proposition}
	\begin{proof}
		$V(\Lambda_0)$ and $V(\Lambda_1)$ are irreducible highest weights representations and $\pi$ preserves triangular decomposition. To finish the proof we notice that action of $\pi$ interchange the highest weights of the representations.
	\end{proof}
	To determine $\tilde{\pi}$ uniquely we require $\tilde{\pi} \left(|\Lambda_i \rangle \right) = |\Lambda_{1-i} \rangle$.
	\begin{proposition} \label{prop: pi involution}
		Conjugation by involution $\tilde{\pi}$ is expressed as follows
		\begin{align}
		\tilde{\pi} \left(\Phi_+^{(1-i,i)}(z) \right) \tilde{\pi} =& (-q^3)^{\frac12-i} z^{-\frac12} \Phi^{(i,1-i)}_-(z), & \tilde{\pi} \left(\Phi^{(1-i,i)}_-(z) \right) \tilde{\pi}=& (-q^3)^{\frac12-i} z^{\frac12} \Phi_+^{(i,1-i)}(z),\\		
		\tilde{\pi} \left(\Psi_+^{(1-i,i)}(z) \right) \tilde{\pi}=& (-q)^{3\left(\frac12-i \right)} z^{-\frac12} \Psi^{(i,1-i)}_-(z), & \tilde{\pi} \left(\Psi_-^{(1-i,i)}(z) \right) \tilde{\pi}=& (-q)^{3\left(\frac12-i \right)} z^{\frac12}\Psi^{(i,1-i)}_+(z).
		\end{align}
	\end{proposition}
	\begin{proof}[Sketch of a proof]
		One can prove the formulas up to a constant via the intertwining properties \cite[Chapter 6]{JM}
		\begin{align}
		\tilde{\pi} \left(\Phi_+^{(1-i,i)}(z) \right) \tilde{\pi} =& c_{1}^{(i)}  \Phi^{(i,1-i)}_-(z), & \tilde{\pi} \left(\Phi^{(1-i,i)}_-(z) \right) \tilde{\pi}=& c_{1}^{(i)} z \Phi_+^{(i,1-i)}(z),\\	
		\tilde{\pi} \left(\Psi_+^{(1-i,i)}(z) \right) \tilde{\pi}=& c_2^{(i)} \Psi^{(i,1-i)}_-(z), & \tilde{\pi} \left(\Psi_-^{(1-i,i)}(z) \right) \tilde{\pi}=& c_2^{(i)} z \Psi^{(i,1-i)}_+(z),
		\end{align}
		here $c_1^{(i)}$ and $c_2^{(i)}$ are some $z$-dependent scalars.
		Then one can find the constants by comparison with normalization \cite[eq. (6.4), (6.5)]{JM}.
	\end{proof}
	\begin{corollary} \label{corollary: involution} The following relations hold
		\begin{align}
		z^{-1/2} \ff(w/z) \Phi_+(z) \Phi_+(w) =& w^{-1/2} \ff(z/w) \Phi_+(w) \Phi_+(z), \label{eq: interchange phi+ phi+}\\
		z^{-1/2} \pp(w/z) \Psi^*_+(z) \Psi^*_+(w) =& w^{-1/2} \pp(z/w) \Psi^*_+(w) \Psi^*_+(z), \label{eq: interchange psi+ psi+}\\
		z^{1/2} \fp(w/z) \Phi_+(z) \Psi_+^*(w) =& w^{1/2} \fp(z/w) \Psi_+^*(w) \Phi_+(z). \label{eq: interchange phi+ psi+}
		\end{align}
	\end{corollary}
	\begin{proof}
		These relations are obtained from \eqref{eq: interchange phi- phi-}--\eqref{eq: interchange psi- phi-} after conjugation by $\tilde{\pi}$.
	\end{proof}
	\section{Vertex operators relations revisited} \label{Section: JM}
	The main results of this section are R-matrix relations (Theorems \ref{Th: R-matrix for Phi} and \ref{Th: R-matrix for Psi}). One can find these relations without delta-function term in \cite{JM}. In \emph{loc. cit.} parameters of vertex operator are numbers, but in our paper the parameters are formal variables. Although formulas below are close to formulas in \emph{loc. cit.}, they have a different meaning and can be considered as a new result.
	\subsection{R-matrix relations}
	
	R-matrix is an operator on $\mathbb{C}^2 \otimes \mathbb{C}^2$. Let $v_{+}$, $v_-$ be a basis of each $\mathbb{C}^2$. The matrix of this operator with respect to basis $v_+\otimes v_+$, $v_+\otimes v_-$, $v_-\otimes v_+$, $v_-\otimes v_-$ looks as follows
	\begin{equation} \label{eq: Rmatrix formulas}
	R(x)=  \begin{pmatrix}  
	1 & 0 & 0 & 0\\
	0 & \frac{(1-x)q}{1-q^2 x}& \frac{1-q^2}{1-q^2 x} & 0\\
	0 & \frac{(1-q^2)x}{1-q^2 x} & \frac{(1-x)q}{1-q^2 x} & 0\\
	0 & 0 & 0 & 1
	\end{pmatrix}.
	\end{equation}
	This R-matrix is an important object in the representation theory of $\qsl$, though in this paper we will not use any information on R-matrix apart from \eqref{eq: Rmatrix formulas}. Below we will see that R-matrix encodes certain interchanging relations for vertex operators.
	\subsubsection{Interchanging relation on $\Phi$-vertex operators}
	Denote by $\delta(x,y)= \sum_{k+l=-1} x^k y^l $.
	\begin{proposition} \label{Th: PhiPhi rel}
		Following relations hold
		\begin{multline}
		z^{-\frac12} \ff(w/z)  \Phi_-(z) \Phi_+ (w) - w^{-\frac12} \ff(z/w) \left( \frac{q(1-z/w)}{q^2-z/w}  \Phi_+ (w) \Phi_- (z) + \frac{(q^2-1)z/w}{q^2-z/w}  \Phi_- (w) \Phi_+ (z) \right) \\=(-1)^{\partial} (-q^3)^{\frac12} q^{-2}  \delta(z,q^2 w),
		\end{multline}
		\vspace{-0.6cm}
		\begin{multline} 
		z^{-\frac12} \ff(w/z)  \Phi_+(z) \Phi_- (w) - w^{-\frac12} \pp(z/w) \left( \frac{q(1-z/w)}{q^2-z/w}  \Phi_- (w) \Phi_+ (z) + \frac{q^2-1}{q^2-z/w}  \Phi_+ (w) \Phi_- (z) \right) \\= -(-1)^{\partial} (-q^3)^{\frac12} q^{-3}  \delta(z,q^2 w).
		\end{multline} 
	\end{proposition}
	\begin{proof}
		Using \eqref{eq: Phi min x_0}, we obtain
		
		\begin{align}
		[\Phi_- (z) \Phi_- (w), x_0^-]_{q^2} = \Phi_- (z) [\Phi_- (w), x_0^- ]_q + q [\Phi_-(z),& x_0^-]_q \Phi_- (w) \notag \\&=\Phi_- (z) \Phi_+(w) + q \Phi_+ (z) \Phi_- (w),
		\\
		[ \Phi_- (z) \Phi_- (w), x_1^- ]_{q^{-2}} =  \Phi_- (z) [\Phi_- (w), x_1^- ]_{q^{-1}} + q^{-1} [&\Phi_- (z), x_1^- ]_{q^{-1}} \Phi_- (w)\notag
		\\&=q^2 w \Phi_- (z) \Phi_+(w) + q z \Phi_+ (z) \Phi_- (w).
		\end{align}
		Solving the system of two linear equations, one can find
		\begin{align}
		\Phi_-(z) \Phi_+ (w) &= \frac{ z[ \Phi_- (z) \Phi_- (w), x_0^-]_{q^2} - [ \Phi_- (z) \Phi_- (w), x_1^- ]_{q^{-2}}}{z(1-q^2 w/z)},\\
		\Phi_+ (z) \Phi_- (w) &= \frac{-qw[ \Phi_- (z) \Phi_- (w), x_0^-]_{q^2} + q^{-1} [ \Phi_- (z) \Phi_- (w), x_1^- ]_{q^{-2}}}{z(1-q^2 w/z)}.
		\end{align}
		Using \eqref{eq: Phi Phi norm ord}, we see that			
		\begin{align}
		(-q^3 z)^{-\frac12} \ff(w/z) \Phi_-(z) \Phi_+ (w) &= \frac{ z[ : \! \Phi_-(z) \Phi_-(w) \! : \ , x_0^-]_{q^2} - [: \! \Phi_-(z) \Phi_-(w) \! : \ ,  x_1^- ]_{q^{-2}}}{z(1-q^2 w/z)}, \label{eq: reg exp for PhiM PhiP} \\
		(-q^3 z)^{-\frac12} \ff(w/z) \Phi_+ (z) \Phi_- (w)  &= \frac{-qw[ : \! \Phi_-(z) \Phi_-(w) \! : \ , x_0^-]_{q^2} + q^{-1} [ : \! \Phi_-(z) \Phi_-(w) \! : \ , x_1^- ]_{q^{-2}}}{z(1-q^2 w/z)}.  \label{eq: reg exp for PhiP PhiM} 
		\end{align}
		
		Then
		\begin{multline} \label{eq: R-matrix Fi proof1}
		z^{-\frac12} \ff(w/z) \Phi_-(z) \Phi_+ (w) - w^{-1/2} \ff(z/w) \left( \frac{q(1-z/w)}{q^2-z/w}  \Phi_+ (w) \Phi_- (z) + \frac{(q^2-1)z/w}{q^2-z/w} \Phi_- (w) \Phi_+ (z) \right) \\
		=(-q^3)^{\frac12} \Big(  q^2 w[ : \! \Phi_- (q^2 w) \Phi_- (w) \! :, x_0^-]_{q^2} - [ : \! \Phi_- (q^2 w) \Phi_- (w) \! :, x_1^- ]_{q^{-2}} \Big)   \delta(z,q^2 w),
		\end{multline}
		\begin{multline} \label{eq: R-matrix Fi proof2}
		z^{-\frac12} \ff(w/z) \Phi_+(z) \Phi_- (w) - w^{-\frac12} \ff(z/w) \left( \frac{q(1-z/w)}{q^2-z/w}  \Phi_- (w) \Phi_+ (z) + \frac{q^2-1}{q^2-z/w}  \Phi_+ (w) \Phi_- (z) \right) \\
		=(-q^3 )^{\frac12} \Big(  -q w[ : \! \Phi_- (q^2 w) \Phi_- (w) \! :, x_0^-]_{q^2} + q^{-1} [ : \! \Phi_- (q^2 w) \Phi_- (w) \! :, x_1^- ]_{q^{-2}} \Big)   \delta(z,q^2 w).
		\end{multline}

		\begin{lemma} \label{lemma: special points}
			Following relation holds
			\begin{equation}
			q^2 w[ : \! \Phi_- (q^2 w) \Phi_- (w) \! :, x_0^-]_{q^2} - [ : \! \Phi_- (q^2 w) \Phi_- (w) \! :, x_1^- ]_{q^{-2}} = (-1)^{\partial} q^{-2}.
			\end{equation}
		\end{lemma}
		
		\begin{proof}
			We will proof the lemma assuming $w$ to be a number, not a formal variable; formal variable version follows. Let us consider two contours of integration $C_+ = \{y \mid |y|=R_+ \gg |w|  \}$, $C_- = \{y \mid |y|=R_- \ll |w|  \}$
			
			Denote $\Omega(w) = \ :\!\Phi_- (q^2 w) \Phi_- (w)\!:$. Note that
			\begin{align}
			[ \Omega(w) , x_0^-]_{q^2} =& \int_{C_-} \Omega(w) X^-(y) dy - q^2 \int_{C_+} X^-(y) \Omega(w)dy,\\
			[ \Omega(w) , x_{1}^-]_{q^{-2}} =& \int_{C_-} y \Omega(w) X^-(y) dy - q^{-2}\int_{C_+} y X^-(y) \Omega(w)dy.
			\end{align}
			Hence
			\begin{multline*}
			q^2 w[ \Omega(w), x_0^-]_{q^2} - [ \Omega(w), x_1^- ]_{q^{-2}} =
			\int_{C_-}  \big( q^2 w -  y \big) \Omega(w) X^-(y) dy -\int_{C_+}  \big( q^4 w - q^{-2} y \big) X^-(y) \Omega(w)dy\\
			=\int_{C_-}  \frac{\big( q^2 w -  y \big)}{q^2 (q^4 w-y)(q^2 w - y)} :\!\Omega(w) X^-(y)\!:dy -\int_{C_+}  \frac{ q^4 w - q^{-2} y }{(y -q^4 w)(y-q^6 w)} :\!X^-(y) \Omega(w)\!:dy\\
			=\int_{C_+}  \frac{ q^{-2}  }{y -q^4 w} :\!X^-(y) \Omega(w)\!:dy - \int_{C_-}  \frac{q^{-2}}{y-q^4 w} :\!\Omega(w) X^-(y)\!:dy   
			\\ =\res_{y=q^4 w} \frac{ q^{-2}  }{y -q^4 w} : \! X^-(y) \Omega(w)\!:dy.
			= (-1)^{\partial} q^{-2}.
			\end{multline*}
			Here we used  $:\!X^-(q^4 w) \Omega(w)\!: =(-1)^{\partial}$.
		\end{proof}
		To finish the proof of Proposition \ref{Th: PhiPhi rel} we apply Lemma \ref{lemma: special points} to \eqref{eq: R-matrix Fi proof1} and \eqref{eq: R-matrix Fi proof2}.
	\end{proof}
	
	\paragraph{Matrix notation} Denote $\Phi(z)=\Phi_+ (z)\otimes v_+ + \Phi_-(z) \otimes v_- \in \Hom(V(\Lambda_i), V(\Lambda_{1-i}) ) \otimes \mathbb{C}^2$. Denote products 
	\[\Phi^{(1)}(z) \Phi^{(2)} (w) = \sum_{\epsilon_1, \epsilon_2 = \pm} \Phi_{\epsilon_1}(z) \Phi_{\epsilon_2}(w) \otimes v_{\epsilon_1} \otimes v_{\epsilon_2},\;\;\; \Phi^{(2)}(w) \Phi^{(1)} (z) = \sum_{\epsilon_1, \epsilon_2 = \pm} \Phi_{\epsilon_2}(w) \Phi_{\epsilon_1}(z) \otimes v_{\epsilon_1} \otimes v_{\epsilon_2}.\] Finally denote by $R^{-1}(z/w) \Phi^{(2)} (w) \Phi^{(1)}(z)$ the result of the action of $R^{-1}(z/w)$ on $\mathbb{C}^2 \otimes \mathbb{C}^2$ tensor multiple of  $\Hom(V(\Lambda_i), V(\Lambda_{i}) ) \otimes \mathbb{C}^2\otimes \mathbb{C}^2$.
	
	\begin{theorem} \label{Th: R-matrix for Phi}
		The following relation holds
		\begin{multline}
		z^{-\frac12} \ff(w/z)  \Phi^{(1)}(z) \Phi^{(2)} (w) = \\ w^{-\frac12} \ff(z/w)  R^{-1}(z/w) \Phi^{(2)} (w) \Phi^{(1)}(z) + (-1)^{\partial}  (-q)^{\frac12} (q^{-1} v_- \otimes v_+ - q^{-2} v_+ \otimes v_-) \delta(z,q^2 w). 
		\end{multline}
	\end{theorem}
	\begin{proof}
		The theorem is just a reformulation of \eqref{eq: interchange phi- phi-}, \eqref{eq: interchange phi+ phi+} and Proposition \ref{Th: PhiPhi rel}.
	\end{proof}
	\subsubsection{Interchanging relation for $\Psi$-vertex operators}
	\begin{proposition} \label{Prop: R-matrix rel for Psi}
		The following relations hold
		\begin{multline}
		z^{-1/2} \pp(w/z)  \Psi_+(z) \Psi_-(w) - 
		w^{-1/2} \pp(z/w)  \left( \frac{q(1-z/w)}{1 - q^2 z/w} \Psi_-(w) \Psi_+(z) + \frac{1-q^2}{1-q^2 z/w} \Psi_+(w) \Psi_-(z) \right)\\=(-1)^{\partial} (-q)^{\frac12} q^2 \delta(q^2z, w),
		\end{multline}
		\begin{multline}
		z^{-1/2}  \pp(w/z) \Psi_-(z) \Psi_+(w) - w^{-1/2} \pp(z/w) \left( \frac{q(1-\frac{z}{w})}{1-q^2 \frac{z}{w}} \Psi_+(w) \Psi_-(z) +\frac{(1-q^2)\frac{z}{w}}{1-q^2 \frac{z}{w}} \Psi_-(w) \Psi_+(z) \right)\\= -(-1)^{\partial}  (-q)^{\frac12} q \delta(q^2z, w).
		\end{multline}
	\end{proposition}
	\begin{proof}
		Using \eqref{eq: Psi min def} we obtain
		\begin{equation*}
		[\Psi_+(z) \Psi_+(w), x_0^+]_{q^2} = \Psi_+(z) [\Psi_+(w), x_0^+]_{q} + q[\Psi_+(z), x_0^+]_q \Psi_+(w) = \Psi_+(z) \Psi_-(w) + q \Psi_-(z) \Psi_+(w),
		\end{equation*}
		\vspace{-2.5em}
		\begin{multline*}
		[\Psi_+(z) \Psi_+(w), x_{-1}^+]_{q^{-2}} = \Psi_+(z) [\Psi_+(w), x_{-1}^+]_{q^{-1}} + q^{-1}[\Psi_+(z), x_{-1}^+]_{q^{-1}} \Psi_+(w) \\ =(q^2 w)^{-1} \Psi_+(z) \Psi_-(w)+ (q^3 z)^{-1} \Psi_-(z) \Psi_+(w).
		\end{multline*}
		Solving the system of linear equations, we obtain
		\begin{align}
		\Psi_+(z) \Psi_-(w) = \frac{w/z \left(- [\Psi_+(z) \Psi_+(w), x_0^+]_{q^2} + q^4 z [:\!\Psi_+(z) \Psi_+(w)\!:, x_{-1}^+]_{q^{-2}} \right)}{q^2 -w/z},\\
		\Psi_-(z) \Psi_+(w) = \frac{q[\Psi_+(z) \Psi_+(w), x_0^+]_{q^2}  -  q^3 w [:\!\Psi_+(z) \Psi_+(w), x_{-1}^+]_{q^{-2}} }{q^2 - w/z}.
		\end{align}
		Using \eqref{eq: Psi Psi norm ord}, we see that
		\begin{align}
		(-qz)^{-1/2} \pp(w/z) \Psi_+(z) \Psi_-(w) =& \frac{ w \left(- [:\!\Psi_+(z) \Psi_+(w)\!:, x_0^+]_{q^2} + q^4 z [:\!\Psi_+(z) \Psi_+(w)\!:, x_{-1}^+]_{q^{-2}} \right)}{q^2 z\left(1 -  \frac{w}{q^2 z} \right)}, \label{eq: reg exp PsiP PsiM} \\
		(-qz)^{-1/2} \pp(w/z) \Psi_-(z) \Psi_+(w) =& \frac{q[:\!\Psi_+(z) \Psi_+(w)\!:, x_0^+]_{q^2}  -  q^3 w [:\!\Psi_+(z) \Psi_+(w)\!:, x_{-1}^+]_{q^{-2}} }{q^2 \left(1 - \frac{w}{q^2 z} \right)}. \label{eq: reg exp PsiM PsiP} 
		\end{align}
		Then
		\begin{multline}
		z^{-\frac12} \pp(w/z)  \Psi_+(z) \Psi_-(w) - w^{-\frac12} \pp(z/w) \left( \frac{1-q^2}{1-q^2 z/w} \Psi_+(w) \Psi_-(z) +\frac{q(1-z/w)}{1 - q^2 z/w} \Psi_-(w) \Psi_+(z) \right) \\ = (-q)^{\frac12} q^2 z \Big(\!\!-[:\!\Psi_+(z) \Psi_+(q^2 z)\!:, x_0^+]_{q^2} + q^4 z [:\!\Psi_+(z) \Psi_+(q^2 z)\!:, x_{-1}^+]_{q^{-2}} \Big) \ \delta(q^2 z, w).
		\end{multline}
		\begin{multline}
		z^{-\frac12} \pp(w/z) \Psi_-(z) \Psi_+(w) -  w^{-\frac12} \pp(z/w) \left( \frac{q(1-z/w)}{1-q^2 z/w} \Psi_+(w) \Psi_-(z) +\frac{(1-q^2)z/w}{1-q^2 z/w} \Psi_-(w) \Psi_+(z) \right) \\ =(-q)^{\frac12} q z\Big( [:\!\Psi_+(z) \Psi_+(q^2 z)\!:, x_0^+]_{q^2}  -  q^4 z [:\!\Psi_+(z) \Psi_+(q^2 z)\!:, x_{-1}^+]_{q^{-2}} \Big) \delta(q^2 z, w).
		\end{multline}
		\begin{lemma} \label{lemma: special point for Psi}
			Following relation holds
			\begin{align}
			[:\!\Psi_+(z) \Psi_+(q^2 z)\!:, x_0^+]_{q^2}  -  q^4 z [:\!\Psi_+(z) \Psi_+(q^2 z)\!:, x_{-1}^+]_{q^{-2}} = - (-1)^{\partial} z^{-1}.
			\end{align}
		\end{lemma}
		\begin{proof}
			Denote by
			\begin{equation}
			\Upsilon(z) =\ :\!\Psi_+(z) \Psi_+(q^2 z)\!:.
			\end{equation} 
			Note that
			\begin{align}
			[\Upsilon(z), x_0^+]_{q^2} &= \int_{C_-} \Upsilon(z) X^+(y)dy- q^2\int_{C_+}  X^+(y)\Upsilon(z)dy,\\
			[\Upsilon(z), x_{-1}^+]_{q^{-2}} &= \int_{C_-} y^{-1} \Upsilon(z) X^+(y)dy - q^{-2}\int_{C_+} y^{-1} X^+(y)\Upsilon(z)dy.
			\end{align}
			Hence
			\begin{gather*}
			[\Upsilon(z), x_0^+]_{q^2}  -  q^4 z [\Upsilon(z), x_{-1}^+]_{q^{-2}} = 
			\int_{C_-} (1-q^4 z/y)\Upsilon(z) X^+(y) dy - q^2 \int_{C_+}  (1 - z/y)X^+(y)\Upsilon(z)dy\\=
			\int_{C_-} \frac{q^2(1-q^4 z/y)}{(q^2 z-y)(q^4z-y)} :\!\Upsilon(z) X^+(y)\!:dy - q^2 \int_{C_+}  \frac{(1 - z/y)}{(y-z)(y-q^2 z)} :\!X^+(y)\Upsilon(z)\!:dy \\=-\int_{C_-} \frac{q^2}{y(q^2 z-y)} :\!\Upsilon(z) X^+(y)\!:dy + \int_{C_+}  \frac{q^2}{y(q^2 z-y)} :\!X^+(y)\Upsilon(z)\!:dy 
			\\
			=\res_{y=q^2z} \frac{q^2}{y(q^2 z-y)} :\!X^+(y)\Upsilon(z)\!:dy = - (-1)^{\partial} z^{-1}.
			\end{gather*}
			Here we used $:\!X^+(q^2 z)\Upsilon(z)\!: \ =(-1)^{\partial}$. \end{proof}
	\end{proof}
	In terms of the operators $\Psi^*$, Proposition \ref{Prop: R-matrix rel for Psi} can be rewritten as follows.
	\begin{corollary} \label{corol: interchange for psi}
		The following relation holds
		\begin{multline}
		z^{-1/2} \pp(w/z)  \Psi^*_-(z) \Psi^*_+(w)  
		-w^{-1/2} \pp(z/w)  \left( \frac{q(1-z/w)}{1 - q^2 z/w} \Psi^*_+(w) \Psi^*_-(z) + \frac{1-q^2}{1-q^2 z/w} \Psi^*_-(w) \Psi^*_+(z) \right) \\= (-1)^{\partial}(-q)^{\frac12} q \delta(q^2z, w),
		\end{multline}
		\begin{multline}
		z^{-1/2}  \pp(w/z) \Psi^*_+(z) \Psi^*_-(w) - w^{-1/2} \pp(z/w) \left( \frac{q(1-\frac{z}{w})}{1-q^2 \frac{z}{w}} \Psi^*_-(w) \Psi^*_+(z) +\frac{(1-q^2)\frac{z}{w}}{1-q^2 \frac{z}{w}} \Psi^*_+(w) \Psi^*_-(z) \right)\\ = - (-1)^{\partial} (-q)^{\frac12} \delta(q^2z, w).
		\end{multline}		
	\end{corollary}
	
	\paragraph{Matrix notation} Denote $\Psi^*(z)=\Psi^*_+ (z)\otimes v^*_+ + \Psi^*_-(z) \otimes v^*_- \in \Hom(V(\Lambda_i), V(\Lambda_{1-i}) ) \otimes \left( \mathbb{C}^2 \right)^*$. Let us emphasise that $\left( \mathbb{C}^2 \right)^*$ is the dual space to $\mathbb{C}^2$, considered in the definition of $\Phi(z)$.
	
	Denote products 
	\begin{align}
	\Psi^{*, (1)}(z) \Psi^{*, (2)} (w) = \sum_{\epsilon_1, \epsilon_2 = \pm 1} \Psi^*_{\epsilon_1}(z) \Psi^*_{\epsilon_2}(w) \otimes v^*_{\epsilon_1} \otimes v^*_{\epsilon_2}, \\
	\Psi^{*, (2)}(w) \Psi^{*, (1)} (z) = \sum_{\epsilon_1, \epsilon_2 = \pm 1} \Psi^*_{\epsilon_2}(w) \Psi^*_{\epsilon_1}(z) \otimes v^*_{\epsilon_1} \otimes v^*_{\epsilon_2}.
	\end{align}
	Finally denote by $ \Psi^{*,(2)} (w) \Psi^{*,(1)}(z) R(z/w)$ the result of the dual action of $R(z/w)$ on $\left( \mathbb{C}^2 \otimes \mathbb{C}^2 \right)^*$ tensor multiple of  $\Hom(V(\Lambda_i), V(\Lambda_{i}) ) \otimes \left( \mathbb{C}^2  \otimes \mathbb{C}^2 \right)^*$. In other words, we multiply the operator-valued \emph{row vector} on the matrix.
	\begin{theorem} \label{Th: R-matrix for Psi}
		The following relation holds
		\begin{multline}
		z^{-\frac12} \pp(w/z)  \Psi^{*,(1)}(z) \Psi^{*,(2)} (w) \\ = w^{-\frac12} \pp(z/w)  \Psi^{,*(2)} (w) \Psi^{*,(1)}(z) R(z/w) + (-1)^{\partial} (-q)^{\frac12} (q v^*_- \otimes v^*_+ -  v^*_+ \otimes v^*_-) \delta(q^2 z, w).
		\end{multline}
	\end{theorem}
	\begin{proof}
		The theorem is just a reformulation of \eqref{eq: interchange psi- psi-}, \eqref{eq: interchange psi+ psi+} and Corollary \ref{corol: interchange for psi}.
	\end{proof}
	
	\subsection{Special point relation}
	\subsubsection{Special point for $\Phi$}
	\begin{proposition} \label{prop: spec point phi}
		We have the following identity
		\begin{equation} \label{eq: special point phi}
		\left. (-q z)^{-1/2} \ff(w/z) \left( q \Phi_-(z) \Phi_+ (w)  - \Phi_+ (z) \Phi_- (w) \right) \right|_{w=q^2z} =  \frac{(-1)^{\partial}}{z q^2 (1-q^2)}.
		\end{equation}
	\end{proposition}
	\begin{remark}
		Note, that \emph{a priori} LHS of \eqref{eq: special point phi} is not well defined since a coefficient of any power of $z$ is an infinite sum of operators. So we have to prove that the result of substitution exists as well as to find the result. Also note, that we substitute $w \mapsto q^2 z$ to the whole expression, not to the individual multiples; the result of substitution to the individual multiples does not have to exist.
	\end{remark}
	\begin{proof}
		Substituting $w \mapsto q^2z$ to \eqref{eq: reg exp for PhiM PhiP} and \eqref{eq: reg exp for PhiP PhiM} we obtain
		\begin{multline}
		\left. (-q^3 z)^{-1/2} \ff(w/z) \left( q \Phi_-(z) \Phi_+ (w)  - \Phi_+ (z) \Phi_- (w) \right) \right|_{w=q^2 z} =  \\ \frac{  q^2 z[ : \! \Phi_-(z) \Phi_-(q^2 z) \! : \ , x_0^-]_{q^2} - [: \! \Phi_-(z) \Phi_-(q^2 z) \! : \ ,  x_1^- ]_{q^{-2}}}{qz(1-q^2)}.
		\end{multline}
		To finish the proof we apply Lemma \ref{lemma: special points}.
	\end{proof}
	\subsubsection{Special point for $\Psi$}
	\begin{proposition} We have the following identity
		\begin{equation}
		\left. (-qz)^{-1/2} \pp(w/z)  \left( \Psi_-(z) \Psi_+ (w) - q \Psi_+(z) \Psi_- (w) \right) \right|_{z=q^2w} =  \frac{q w^{-1} }{1-q^2} (-1)^{\partial}.
		\end{equation}
	\end{proposition}
	\begin{proof}
		Let us substitute $z \mapsto q^2 w$ to \eqref{eq: reg exp PsiP PsiM} and \eqref{eq: reg exp PsiM PsiP} 
		\begin{multline}
		\left. (-qz)^{-1/2} \pp(w/z)  \left( \Psi_-(z) \Psi_+ (w) - q \Psi_+(z) \Psi_- (w) \right) \right|_{z=q^2 w} =  \\
		\frac{q [  : \! \Psi_+ (q^2 w) \Psi_+ (w) \! : , x_0^+]_{q^2} - q^5 w [ : \! \Psi_+ (q^{2} w) \Psi_+ (w) \! : , x_{-1}^+ ]_{q^{-2}} }{q^2-1 }.
		\end{multline}
		To finish the proof we applying Lemma \ref{lemma: special point for Psi}.
	\end{proof}
	
	\begin{corollary} \label{corol: special point for Psi}
		For any $q_1 \in \mathbb{C} \backslash \{0\}$ we have the following identity
		\begin{equation}
		\left. (-q_1/z)^{1/2} \pp(w/z)  \left( \Psi^*_+(q q_1 z) \Psi^*_- (q q_1 w) - q \Psi^*_-(q q_1 z) \Psi^*_+ (q q_1 w) \right) \right|_{z=q^2w} = - \frac{  w^{-1} }{1-q^2} (-1)^{\partial}.
		\end{equation}
	\end{corollary}
	
	\subsection{Interchanging relation on $\Phi$ and $\Psi$}
	\begin{proposition} \label{prop: interchange Phi and Psi}
		The following relation holds
		\begin{align}
		z^{\frac12} \fp(w/z) \Phi_{\epsilon_1} (z) \Psi^{*}_{\epsilon_2} (w) = w^{\frac12} \fp(z/w) \Psi^{*}_{\epsilon_2} (w)  \Phi_{\epsilon_1} (z).
		\end{align}
	\end{proposition}
	\begin{proof}
		We have already seen the cases  $\epsilon_1= \epsilon_2= \pm$, see \eqref{eq: interchange psi- phi-} and \eqref{eq: interchange phi+ psi+}. To prove remaining cases, let us use a relation from \cite[Section 6.3]{JM} and relation \cite[(6.12)]{JM}
		\begin{align}
		\Psi_-^*(z) x_0^- -  x_0^- \Psi_-^*(z) =&0, & \Phi_-(z) x_0^+ -  x_0^+ \Phi_-(z) =&0.
		\end{align}
		To be combined with \eqref{eq: Phi min x_0} and \eqref{eq: Psi min def}, the relations yield
		\begin{align}
		[\Phi_-(z) \Psi_-^*(w), x^-_0 ]_q = [\Phi_-(z), x^-_0 ]_q \Psi_-^*(w) =\Phi_+(z) \Psi_-^*(w),\\
		[\Phi_-(z) \Psi_-^*(w), x^+_0 ]_q= \Phi_-(z) [\Psi_-^*(w), x^+_0 ]_q  = \Phi_-(z) \Psi_+^*(w).
		\end{align}
		Considering  $q$-commutator of \eqref{eq: interchange psi- phi-} with $x_0^-$ and $x_0^+$, we obtain cases $\epsilon_1=+$, $\epsilon_2=-$ and $\epsilon_1=-$, $\epsilon_2=+$ correspondingly. 
	\end{proof}
	
	\section{Realization of (Twisted) Deformed Virasoro algebra} \label{sec: Vir}
	In this section, we will consider two algebras: \emph{deformed Virasoro algebra} and \emph{twisted deformed Virasoro algebra}. Deformed Virasoro algebra is extensively studied. Twisted Virasoro was defined in \cite{Sh}, though this algebra is considerably less famous.
	
	The algebras depend on two parameters $q_1$, $q_2$. It is also convenient to consider $q_3$ such that $q_1 q_2 q_3 = 1$. In this section we study connection between the algebras and $\qsl$ for $q^2 = q_3$.
	
	To define (twisted) deformed Virasoro algebra, we need the following notation
	\begin{equation}
	\sum_{l=0}^{\infty} f_l x^l = f(x) = \exp \left( \sum_{n=1}^{\infty} \frac{1}{n} \frac{(1-q_1^{n})(1-q_2^{n})}{1+q_3^{-n}} x^n\right).
	\end{equation}
	Note that
	\begin{equation} \label{eq: f via betta}
	f(x) = \frac{1}{1-x} \fp\left(q_1 q x \right) \fp \left( q_1^{-1} q^{-1} x \right).
	\end{equation}
	\subsection{Deformed Virasoro algebra}
	\begin{definition} Deformed Virasoro algebra $\Vir$ is generated by $T_n$ for $n \in \mathbb{Z}$. The defining relation is
		\begin{equation} \label{defeq: Vir}
		\sum_{l=0}^{\infty}f_l T_{n-l} T_{m+l} - \sum_{l=0}^{\infty}f_l T_{n-l} T_{m+l} = -\frac{(1-q_1)(1-q_2)}{1-q_3^{-1}} \left( q_3^{-n} - q_3^{n} \right) \delta_{n+m,0}.
		\end{equation}
		Denote $T(z)= \sum_{n\in  \mathbb{Z}} T_n z^{-n}$, $\delta(x) = \sum_{k \in \mathbb{Z}} x^k$.
		Relation \eqref{defeq: Vir} is equivalent to
		\begin{equation} \label{defeq: Vir current}
		f(w/z) T(z) T(w) - f(z/w) T(w) T(z) = -\frac{(1-q_1)(1-q_2)}{1-q_3^{-1}} \left( \delta\left(\frac{w}{q_3 z}\right) -  \delta\left(\frac{q_3 w}{z}\right) \right).
		\end{equation}	
	\end{definition}
	\paragraph{Representation} Recall that $V_j$ were defined by \eqref{notatio: Vj}.
	\begin{theorem} \label{th: Vir non twist}
		The formula below determines an action of $\Vir$ on $V_j$ for all $j \in \mathbb{Z}$.
		\begin{equation} \label{eq: realization Vir non twist}
		T(z) =  z^{1/2} \frac{q^{3/2}(q_1^{1/2}- q_1^{-1/2})}{\fp(q/q_1)}  \Big(u \Psi_{+}^* (q q_{1} z)  \Phi_{+} (z)  + u^{-1} \Psi_{-}^* (q q_1 z) \Phi_{-} ( z) \Big).
		\end{equation}
	\end{theorem}
	Denote the obtained representation by $\RepV$.
	
	\begin{remark} \label{remark on second term}
		A bosonization of deformed Virasoro algebra is known since \cite{SKAO}, but our bosonization is a different one. In both cases current $T(z)$ is presented as a sum two summands. Surprisingly, the first summands in both cases are `the same normally ordered exponent of Heisenberg $a_k$'; however, the second ones are different. In \cite{SKAO} the second summand is also an exponent of the same Heisenberg, but this is not true for our bosonization. Note that  $\Psi_{-}^* (q q_1 z) \Phi_{-} ( z) $ is an exponent of Heisenberg $\pi(a_k)$, but not of $a_k$.
	\end{remark}
	
	\begin{proof}
		The proof is basically verification of \eqref{defeq: Vir current}. Let us rewrite \eqref{eq: realization Vir non twist} in the matrix form
		\begin{equation}
		T(z) = \frac{q^{3/2}(q_1^{1/2}- q_1^{-1/2})}{\fp(q/q_1)} \Psi^*(q q_1  z)  \varepsilon_z \Phi(z), \text{ for } \varepsilon_z = z^{1/2} \begin{pmatrix} u & 0 \\ 0 & u^{-1} \end{pmatrix}.
		\end{equation}
		Using \eqref{eq: f via betta}, Proposition \ref{prop: interchange Phi and Psi} and \eqref{eqdef: alpha and beta} we obtain
		\begin{multline} \label{eq: calucaltion1}
		f(w/z)  \left( \Psi^*(q q_1  z)  \varepsilon_z \Phi(z) \right) \left( \Psi^*(q q_1  w) \varepsilon_w \Phi(w) \right) \\= 
		\frac{1}{1-w/z} \fp\left(\frac{q_1 q w}{z} \right) \fp \left(\frac{w}{q_1 q z}  \right) \left( \Psi^*(q q_1 z)  \varepsilon_z \Phi(z) \right) \left( \Psi^*(q q_1 w) \varepsilon_w \Phi(w) \right) \\= 
		\frac{1}{1-w/z} \left(\frac{q q_1 w}{z} \right)^{\frac12} \fp\left(\frac{z}{q_1 q w} \right) \fp \left(\frac{w}{q_1 q z} \right) \left( \Psi^{*, (1)}(q q_1  z)  \Psi^{*,(2)}(q q_1 w) \right) \varepsilon_z \otimes \varepsilon_w \left(\Phi^{(1)}(z)  \Phi^{(2)}(w) \right)\\=
		\fp\left(\frac{z}{q_1 q w} \right) \fp \left(\frac{w}{q_1 q z} \right) (q q_1 z w)^{\frac12}\\  \times\left(z^{-\frac12} \pp \left(\frac{w}z \right) \Psi^{*, (1)}(q q_1  z)  \Psi^{*,(2)}(q q_1 w) \right)  \varepsilon_z \otimes \varepsilon_w \left( z^{-\frac12} \ff \left(\frac{w}z \right) \Phi^{(1)}(z)  \Phi^{(2)}(w) \right).
		\end{multline}
		To continue the calculation, we apply Theorems \ref{Th: R-matrix for Phi} and \ref{Th: R-matrix for Psi}. The RHS of \eqref{eq: calucaltion1} can be 
		presented as sum of three summands. The first summand is
		\begin{multline} \label{eq: calculation2}
		\fp\left(\frac{z}{q_1 q w} \right) \fp \left(\frac{w}{q_1 q z} \right) (q q_1 z w)^{\frac12} \\ \times \left(w^{-\frac12} \pp(z/w)  \Psi^{*,(2)}(q q_1 w) \Psi^{*, (1)}(q q_1  z) \right) R(z/w) \varepsilon_z \otimes \varepsilon_w R^{-1}(z/w) \left( w^{-\frac12} \ff(z/w) \Phi^{(2)}(w) \Phi^{(1)}(z) \right)\\=
		\frac{1}{1-z/w} \left(\frac{q q_1 z }{w} \right)^{\frac12} \fp\left(\frac{z}{q_1 q w} \right) \fp \left(\frac{w}{q_1 q z} \right)   \left( \Psi^{*,(2)}(q q_1 w) \Psi^{*, (1)}(q q_1  z)  \right) \varepsilon_z \otimes \varepsilon_w  \left(\Phi^{(2)}(w) \Phi^{(1)}(z) \right)\\=
		\frac{1}{1-z/w} \fp\left(\frac{z}{q_1 q w} \right) \fp \left(\frac{q_1 q z}{w} \right)   \left( \Psi^{*}(q q_1 w) \varepsilon_w \Phi(w)  \right) \left( \Psi^{*}(q q_1  z) \varepsilon_z \Phi(z) \right)\\=
		f(z/w)  \left( \Psi^{*}(q q_1 w) \varepsilon_w \Phi(w)  \right) \left( \Psi^{*}(q q_1  z) \varepsilon_z \Phi(z) \right).
		\end{multline}
		Here we have used  Proposition \ref{prop: interchange Phi and Psi} and an important property 
		\begin{equation} \label{eq: prop of epsilon}
		R(z/w) \varepsilon_z \otimes \varepsilon_w R^{-1}(z/w) = \varepsilon_z \otimes \varepsilon_w.
		\end{equation}
		The second summand without factor $\fp\left(\frac{z}{q_1 q w} \right) \fp \left(\frac{w}{q_1 q z} \right)$ is
		\begin{gather*}
		(q q_1 z w)^{\frac12}\!\!\left( z^{-\frac12} \pp\left(\frac{w}{z}\right) \Psi^{*,(\hspace{-0.07em}1\hspace{-0.07em})}\!(q q_1  z)  \Psi^{*,(\hspace{-0.07em}2\hspace{-0.07em})}\!(q q_1 w) \right)\!\varepsilon_z\!\otimes \varepsilon_w\!\left(\! (-q)^{\frac12} (q^{-1} v_-{\otimes}v_+{-}q^{-2} v_+{\otimes}v_-) \delta(z,q^2 w)\!\right)\!(-1)^{\partial}\\
		=(q q_1)^{\frac12} zw \left( z^{-\frac12} \pp(w/z) \Psi^{*, (1)}(q q_1  z)  \Psi^{*,(2)}(q q_1 w) \right)  \left( (-q)^{\frac12} (q^{-1} v_-{\otimes}v_+ - q^{-2} v_+{\otimes}v_-) \delta(z,q^2 w) \right)(-1)^{\partial}\\=
		q^{-1}  zw  (-q_1/z)^{\frac12} \pp(w/z)  \Big( q \Psi^{*}_-(q q_1  z)  \Psi^{*}_+(q q_1 w) - \Psi^{*}_+(q q_1  z)  \Psi^{*}_-(q q_1 w)  \Big) (-1)^{\partial} \delta(z,q^2 w)\\=
		q^{-1}  zw \frac{  w^{-1} }{1-q^2} \delta(z,q^2 w)= \frac{ 1 }{q(1-q^2)} \delta(z/q^2 w).
		\end{gather*}
		Here we used Corollary \ref{corol: special point for Psi}.
		
		The third summand  without factor $\fp\left(\frac{z}{q_1 q w} \right) \fp \left(\frac{w}{q_1 q z} \right)$ is
		\begin{gather*}
		q q_1 (zw)^{\frac12} \left( (-q)^{\frac12} (q v^*_- \otimes v^*_+ -  v^*_+ \otimes v^*_-) \delta(q^3 q_1 z, q q_1w)  \right) \varepsilon_z \otimes \varepsilon_w \left( z^{-\frac12} \ff(w/z) \Phi^{(1)}(z)  \Phi^{(2)}(w)  \right) (-1)^{\partial}\\
		=q q_1 zw \left( (-q)^{\frac12} (q v^*_- \otimes v^*_+ -  v^*_+ \otimes v^*_-) \delta(q^3 q_1 z, q q_1w)  \right)  \left( z^{-\frac12} \ff(w/z) \Phi^{(1)}(z)  \Phi^{(2)}(w)  \right) (-1)^{\partial}\\=
		- q^2 q_1  zw \ (-qz)^{-\frac12} \ff(w/z)   \Big( q \Phi_-(z) \Phi_+(w) -  \Phi_+(z)  \Phi_-(w)   \Big) (-1)^{\partial} \delta(q^3 q_1 z, q q_1w) \\=
		-q^2 q_1  zw   \frac{1}{z q^2 (1-q^2)}  \delta(q^3 q_1 z, q q_1w)=- \frac{1}{q (1-q^2)}  \delta(q^2 z/w).
		\end{gather*}
		Here we used Proposition \ref{prop: spec point phi}
		
		When we calculated the second and the third summands, we have omitted the multiple
		\begin{equation}
		\fp\left(\frac{q^2}{q_1 q } \right) \fp \left(\frac{q^{-2}}{q_1 q } \right)   =\frac{1-q^{-2} q_1^{-1}}{1- q_1^{-1}} \left( \fp(q/q_1)  \right)^2 = \frac{1-q_2}{1- q_1^{-1}} \left( \fp(q/q_1)  \right)^2.
		\end{equation}
		So the delta-function coefficient is
		\begin{equation}
		\frac{1-q_2}{1- q_1^{-1}} \left( \fp(q/q_1)  \right)^2 \times \frac{1}{q(1-q^2)}= \frac{(1-q_1)(1-q_2)}{(1-q_3^{-1})}  \left( \frac{\fp(q/q_1)}{q^{3/2}(q_1^{1/2}-q_1^{-1/2})}  \right)^2.
		\end{equation}
		So we have proven
		\begin{multline}
		f(w/z)  \left( \Psi^*(q q_1  z)  \varepsilon_z \Phi(z) \right) \left( \Psi^*(q q_1  w) \varepsilon_w \Phi(w) \right) - f(z/w)  \left( \Psi^{*}(q q_1 w) \varepsilon_w \Phi(w)  \right) \left( \Psi^{*}(q q_1  z) \varepsilon_z \Phi(z) \right) \\=
		\frac{(1-q_1)(1-q_2)}{(1-q_3^{-1})}  \left( \frac{\fp(q/q_1)}{q^{3/2}(q_1^{1/2}-q_1^{-1/2})}  \right)^2  \left( \delta \left( \frac{z}{q^2 w} \right)  -  \delta \left( \frac{q^2 z}{w} \right)  \right).
		\end{multline}
		Evidently, this is equivalent to the theorem.
	\end{proof}
	\paragraph{Connection with Verma module.} \emph{Highest weight vector} $| \lambda \rangle$ for $\Vir$ with \emph{highest weight} $\lambda \in \mathbb{C}$ in a $\Vir$-module is defined by the following properties
	\begin{align} \label{eq: highest weight def}
	T_0 | \lambda \rangle &= \lambda | \lambda \rangle, & T_n |\lambda \rangle&=0 \quad \text{for $n>0$.}
	\end{align}
	Denote by $\jvec = 1 \otimes \mathbb{C} e^{\Lambda_i + \lfloor \frac{j}{2} \rfloor \alpha} \in V_j$ the highest weight vector with respect to Heisenberg algebra.
	\begin{proposition} \label{prop: highest weight non twist}
		Vector $\jvec \in \RepV$ is a highest weight vector for $\Vir$ with the highest weight 
		\begin{equation}\label{eq:lambda uj}
		\lambda_{u,j} = (-q)^{1/2} (q q_1)^{j/2} u + \left( (-q)^{1/2} (q q_1)^{j/2} u \right)^{-1}.
		\end{equation}
	\end{proposition}
	\begin{proof}
		Using \eqref{eq: Psi Phi norm ord}, we obtain
		\begin{multline} \label{eq: Phi Psi weight}
		\frac{q^{3/2}(q_1^{1/2}- q_1^{-1/2})}{\fp(q/q_1)} z^{1/2}\Psi_{-}^* (q q_1 z) \Phi_{-} ( z) = \frac{q^{3/2}(q_1^{1/2}- q_1^{-1/2})}{\fp(q/q_1)} \frac{(-q^3 \times q q_1 z)^{-1/2}}{\fp(1/q q_1)} z^{1/2}:\!\Psi_{-}^* (q q_1 z) \Phi_{-} ( z)\!: \\=
		\frac{q^{3/2}(q_1^{1/2}- q_1^{-1/2})(-q^3 \times q q_1 )^{-1/2}}{(1-1/q_1)}  :\!\Psi_{-}^* (q q_1 z) \Phi_{-} ( z)\!: = (-q)^{-1/2}  :\!\Psi_{-}^* (q q_1 z) \Phi_{-} ( z)\!:.
		\end{multline}
		Using \eqref{eq: Phi Psi weight} and formulas for explicit bosonization \eqref{eq: Phi- boson} and \eqref{eq: Psi-* boson}, we obtain 
		\begin{equation} \label{eq:half of highest weight minus}
		z^{1/2} \frac{q^{3/2}(q_1^{1/2}- q_1^{-1/2})}{\fp(q/q_1)}  \Psi_{-}^* (q q_{1} z)  \Phi_{-} (z) \jvec = (-q)^{-1/2} (qq_1)^{-j/2} \jvec + O(z).
		\end{equation}
		here $O(z)$ is a formal power series, which vanishes at $z=0$, i.e. $\sum_{n>0} \alpha_n z^n$.
		\begin{lemma} \label{pi inv lemma weight}
			Vector $\tilde{\pi} \jvec$ coincides up to a scale with vector $| 1 - j \rangle$.
		\end{lemma}
		\begin{proof}[Sketch of a proof.]
			Let us consider two grading on $V(\Lambda_0) \oplus V(\Lambda_1)$
			\begin{align}
			&\degp \jvec = j(j-1)/4  & &\deg_{K} v = j \quad \text{iff $Kv=q^{j}v$} \label{eq: deg fr1} \\
			&\degp a_{-k} =  k & & \label{eq: deg fr2}
			\end{align}
			One can check that
			\begin{align}
			\deg_{K} \left( \pi(v) \right) &= 1 - \deg_K v&		   \degp \left( \pi(v) \right) &= \deg_{\text{pr}} v
			\end{align}
			Up to a scale, vector $\jvec$ is the only vector with $\deg_K=j$ and $\degp = j(j-1)/4$.
		\end{proof}
		Let us apply $\pi$-involution to \eqref{eq:half of highest weight minus}; Proposition \ref{prop: pi involution} and Lemma \ref{pi inv lemma weight} imply
		\begin{multline} 
		z^{1/2} \frac{q^{3/2}(q_1^{1/2}- q_1^{-1/2})}{\fp(q/q_1)} (-q)^{3 \left( \frac12 - i \right)} \left(q_1 q^3 z\right)^{-\frac12} \Psi_{+}^{(i, 1-i),*} \!(q q_{1} z) \ (-q^3)^{i-\frac12} z^{\frac12}\Phi^{(1-i,i)}_{+} (z) | 1 - j \rangle \\= (-q)^{-1/2} (qq_1)^{-j/2} |1-j \rangle + O(z).
		\end{multline} 
		Replacing of $j \mapsto 1-j$, we obtain 
		\begin{equation} \label{eq:half of highest weight plus}
		z^{1/2} \frac{q^{3/2}(q_1^{1/2}- q_1^{-1/2})}{\fp(q/q_1)} \Psi_{+}^* (q q_{1} z)  \Phi_{+} (z) \jvec = (-q)^{1/2} (qq_1)^{j/2} \jvec + O(z).
		\end{equation}
		Comparison of \eqref{eq:half of highest weight minus} and \eqref{eq:half of highest weight plus} with \eqref{eq: realization Vir non twist} finishes the proof. 
	\end{proof}

	\emph{Verma module} $M(\lambda)$ is a module with cyclic highest weight vector $|\lambda \rangle$ and without any other relations apart from \eqref{eq: highest weight def}.  Verma module $M(\lambda)$ enjoys a universal property: it maps to any module with a highest weight vector of weight $\lambda$.
	According to Proposition \ref{prop: highest weight non twist}, there is a natural map $\phi_{u,j} \colon M(\lambda_{u,j}) \rightarrow \RepV$. We will say that $\lambda$ is generic if $\lambda\neq \pm (q_1^{r/2}q_2^{s/2}+q_1^{-r/2}q_2^{-s/2})$ for $r,s \in \mathbb{Z}_{\geq 1}$.
	
	\begin{proposition}
		For generic $\lambda$ the Verma module $M(\lambda)$ is irreducible. Dimension of $n$th graded component is $p(n)$, i.e. the number of partitions of $n$ elements.
	\end{proposition}
	
	This proposition follows from the fact that determinant of the Shapovalov form for such $\lambda$ is nonzero, this fact was proven in \cite[Th. 3.3]{Bouwknegt:1998}, using \cite{SKAO}. One can also deduce this from the irreducibility of tensor product of Fock modules of toroidal algebra $U_{q_1,q_2,q_3}(\ddot{\mathfrak{gl}}_1)$ \cite[Lem 3.1]{FFJMM} and relations to $W$-algebras \cite{N16} \cite{FHSSY}.
	
	We will say that pair $u,j$ is generic if the corresponding highest weight $\lambda_{u,j}$ is generic.
	
	\begin{corollary}
		For generic values of $u,j$ the module $\RepV$ is irreducible. The natural map $\phi_{u,j} \colon M(\lambda_{u,j}) \rightarrow \RepV$ is an isomorphism.
	\end{corollary}
	\begin{proof}
		Note that dimensions of graded components of both $M(\lambda)$ and $\RepV$  equals to $p(n)$, in particular they coincide. If $M(\lambda_{u,j})$ is irreducible, then the map $\phi_{u,j} \colon M(\lambda_{u,j}) \rightarrow \RepV$ is an isomorphism.
	\end{proof}
	\begin{remark}
		As it was mentioned in Remark \ref{remark on second term}, another bosonization of $\Vir$ was constructed in \cite{SKAO}. Moreover, their formula for the highest weight essentially coincides with our formula \eqref{eq:lambda uj}. Namely, in notation of \cite[Sec. 3]{FF} the highest weight of the representation on the Fock space $\pi_\mu$ equals to $\lambda_{u,j}$ if $q^\mu$ in notation of \emph{loc. cit.} equals to $(-q^3)^{\frac12} (qq_1)^ju$ in notation of this paper (note that parameters $q,p$ in \emph{loc. cit.} correspond to $q_1, q_3^{-1}$ in this paper). For generic $u,j$ these modules are isomorphic since they both are isomorphic to irreducible Verma module.
		
	\end{remark}
	\subsection{Twisted Deformed Virasoro algebra}
	\begin{definition} Twisted deformed Virasoro algebra is generated by $T_r$ for $r \in 1/2+\mathbb{Z}$. The defining relation is
		\begin{equation} \label{defeq: twisted Vir}
		\sum_{l=0}^{\infty}f_l T_{r-l} T_{s+l} - \sum_{l=0}^{\infty}f_l T_{s-l} T_{r+l} = -\frac{(1-q_1)(1-q_2)}{1-q_3^{-1}} \left( q_3^{-r} - q_3^{r} \right) \delta_{r+s,0}.
		\end{equation}
	\end{definition}
	Denote $T(z)= \sum_{r\in 1/2 + \mathbb{Z}} T_r z^{-r}$, $\dood(x) = \sum_{r\in 1/2 + \mathbb{Z}} x^r$
	Relation \eqref{defeq: twisted Vir} is equivalent to
	\begin{equation} 
	f(w/z) T(z) T(w) - f(z/w) T(w) T(z) = -\frac{(1-q_1)(1-q_2)}{1-q_3^{-1}} \left( \dood\left(\frac{w}{q_3 z}\right) -  \dood\left(\frac{q_3 w}{z}\right) \right).
	\end{equation}	
	\begin{theorem} \label{th: Vir twist}
		Formulas below determines an action of Twisted Deformed Virasoro algebra on $V(\Lambda_i)$ for $i=0, 1$
		\begin{equation} \label{eq: realization Vir twist}
		T(z) = (-1)^{1/2}\frac{q^{3/2}(q_1^{1/2}- q_1^{-1/2})}{\fp(q/q_1)}  \left( z \Psi_{-}^* (q q_{1} z)  \Phi_{+} (z)  +  \Psi_{+}^* (q q_1 z) \Phi_{-} ( z) \right).
		\end{equation}
	\end{theorem}
	Denote the obtained representation by $\RepVt$.
	\begin{proof}
		Let us rewrite \eqref{eq: realization Vir twist} in the matrix form
		\begin{equation}
		T(z) = (-1)^{1/2} \frac{q^{3/2}(q_1^{1/2}- q_1^{-1/2})}{\fp(q/q_1)} \Psi^*(q q_1  z)  \varepsilon_z \Phi(z), \text{ for } \varepsilon_z =  \begin{pmatrix} 0 & 1\\
		z & 0 \end{pmatrix}.
		\end{equation}
		The proof is very similar to the proof of Theorem \ref{th: Vir non twist}. A crucial point is that \eqref{eq: prop of epsilon} holds for the new $\varepsilon_z$. Hence RHS of \eqref{eq: calucaltion1} still can be presented as sum of three summands. The first summand is still given by \eqref{eq: calculation2}. The second summand without factor $\fp\left(\frac{z}{q_1 q w} \right) \fp \left(\frac{w}{q_1 q z} \right)$
		\begin{gather*}
		\left(q q_1 wz \right)^{\frac12} \!\left(\! z^{-\frac12}\pp(w/z) \Psi^{*, (1)}(q q_1  z)  \Psi^{*,(2)}(q q_1 w) \!\right)\! \varepsilon_z{\otimes} \varepsilon_w \!\left(\! (-q)^{\frac12} (q^{-1} v_-{\otimes}v_+{-}q^{-2} v_+{\otimes}v_-) \delta(z,q^2 w)\! \right)\!\! (-1)^{\partial}\\=
		\left(- q_1  \right)^{\frac12} \left( z^{-\frac12}\pp(w/z) \Psi^{*, (1)}(q q_1  z)  \Psi^{*,(2)}(q q_1 w) \right) \varepsilon_z \otimes \varepsilon_w \left(  q^{-1} v_-{\otimes}v_+ - q^{-2} v_+{\otimes} v_- \right) \dood(z/q^2 w) (-1)^{\partial}\\
		=\left(-q_1/z \right)^{\frac12}  \pp(w/z) \Psi^{*, (1)}(q q_1  z)  \Psi^{*,(2)}(q q_1 w)   \left( q^{-1} w  v_+{\otimes}v_- - w v_-{\otimes}v_+ \right) \dood(z/q^2 w) (-1)^{\partial}\\
		=q^{-1}w \times  \left(-q_1/z \right)^{\frac12}  \pp(w/z)  \left( \Psi^*_+(q q_1 z) \Psi^*_-(q q_1w) - q \Psi_-(q q_1 z) \Psi^*_+(q q_1w) \right)  \dood(z/q^2 w) (-1)^{\partial}\\=
		-q^{-1}w \frac{w^{-1}}{1-q^2}  \dood(z/q^2 w) =  -\frac{1}{q(1-q^2)}  \dood(z/q^2 w).
		\end{gather*}
		The third summand without factor $\fp\left(\frac{z}{q_1 q w} \right) \fp \left(\frac{w}{q_1 q z} \right)$
		\begin{gather*}
		q q_1 \left(wz \right)^{\frac12} \left( (-q)^{\frac12} (q v^*_- \otimes v^*_+ -  v^*_+ \otimes v^*_-) \delta(q^3 q_1 z, q q_1 w)  \right) \varepsilon_z \otimes \varepsilon_w \left( z^{-\frac12}\ff(w/z) \Phi^{(1)}(z)  \Phi^{(2)}(w) \right) (-1)^{\partial}\\=
		- (- q)^{-\frac12}  \left(  (q v^*_- \otimes v^*_+ -  v^*_+ \otimes v^*_-) \dood (q^2 z/w)  \right) \varepsilon_z \otimes \varepsilon_w \left( z^{-\frac12}\ff(w/z) \Phi^{(1)}(z)  \Phi^{(2)}(w) \right) (-1)^{\partial}\\=
		-(- q)^{-\frac12}  \left(  qz v^*_+ \otimes v^*_- - w v^*_- \otimes v^*_+  \right)  \left( z^{-\frac12}\ff(w/z) \Phi^{(1)}(z)  \Phi^{(2)}(w) \right) \dood (q^2 z/w) (-1)^{\partial}\\=
		qz (-qz)^{-\frac12}\ff(w/z) \left( q  \Phi_-(z) \Phi_+(w) - \Phi_+(z) \Phi_-(w) \right)   \dood (q^2 z/w) (-1)^{\partial}\\=
		qz \frac{1}{z q^2 (1-q^2)}  \dood (q^2 z/w)= \frac{1}{q (1-q^2)}  \dood (q^2 z/w).
		\end{gather*}
		The end of the proof is almost the same as in non-twisted case.
	\end{proof}
	
	\paragraph{Connection with Verma module.}
	\emph{Highest weight vector} $\vac$ in a representation of twisted deformed Virasoro algebra is defined by the following properties
	\begin{align} \label{eq: hw twisted case}
	T_r \vac =0 \quad \text{for $r>0$.}
	\end{align}
	\begin{proposition}
		The vectors $| \Lambda_i \rangle \in \RepVt$ are highest weight vectors.
	\end{proposition}	
	\begin{proof}
		Recall the grading $\degp$ on $V(\Lambda_i)$ defined by \eqref{eq: deg fr1} and \eqref{eq: deg fr2}.  One can check that $\RepVt$ is a graded $\Virt$-module with respect to grading $\degp T_{-r}=r$. To finish the proof one has to note that $\degp |0 \rangle = \degp |1 \rangle=0$ and $\degp \jvec >0$ for $j \neq 0,1$.
	\end{proof}
	
	Verma module $\tVerm$ of twisted Virasoro algebra is a cyclic module with cyclic vector $\vac$ and without any other relations apart from \eqref{eq: hw twisted case}. Verma module enjoys a universal property: it maps to any module with a highest weight vector. Hence there exist a natural map $\phi_i \colon \tVerm \rightarrow \RepVt$ such that $\vac \mapsto | \Lambda_i \rangle $.
	
	\begin{lemma} \label{lemma: prePBW}
		Verma module $\tVerm$ is spanned by 
		\begin{equation} \label{eq: PBW element}
		T_{-r_m} \dots T_{-r_1} \vac \quad \text{for $0 <r_1 \leq r_2 \leq \dots \leq r_m$.}
		\end{equation}
	\end{lemma}
	
	\begin{proof}[Sketch of a proof.]
		One can prove that any element $T_{s_1} \dots T_{s_k} \vac$ can be presented as a linear combination of vectors \eqref{eq: PBW element} using \eqref{defeq: twisted Vir} by induction.
	\end{proof}
	
	\begin{proposition}
		For generic $q_3$ the Verma module $\tVerm$ is irreducible. Natural maps $\phi_i \colon \tVerm \rightarrow \RepVt$ are isomorphisms.
	\end{proposition}
	\begin{proof}
		The representation $\RepVt$ for $q_3=1$ was considered in \cite[Example 7.1]{BG}; it follows from \cite[Section 7.3]{BG} that the representations are irreducible. Hence $\RepVt$ is irreducible for generic $q_3$. Then maps $\phi_i \colon M^{\text{tw}} \rightarrow \RepVt$ are surjective. Now recall the Gauss identity
		\begin{equation} \label{eq: Gauss}
		\prod_{r=\frac12+\mathbb{Z}_{\geq 0}}  \frac{1}{1- \mathfrak{q}^r} =\sum_{j\in i+2\mathbb{Z}} \frac{\mathfrak{q}^{\frac{j(j-1)}{4}}}{\prod_{n=1}^{\infty} (1-\mathfrak{q}^n)} \quad \text{for $i=0,1$.}
		\end{equation}
		According to Lemma \ref{lemma: prePBW}, dimensions of graded components of $\tVerm$ do not exceed corresponding coefficient of LHS of \eqref{eq: Gauss}. On the other hand, coefficients of RHS of \eqref{eq: Gauss} are equal to dimensions of graded components of $\RepVt$. Hence, it follows from surjectivity of $\phi_i$ that $\phi_i$ is an isomorphism.
	\end{proof}	
	\begin{corollary}
		For generic $q_3$ vectors \eqref{eq: PBW element} form a basis of $\tVerm$.
	\end{corollary}
	\appendix

	\bibliographystyle{alpha}
	\bibliography{bibtex}

\newcommand{\etalchar}[1]{$^{#1}$}
\begin{thebibliography}{SKAO96}

\bibitem[BG19]{BG}
M.~Bershtein and R.~Gonin.
\newblock Twisted representations of algebra of q-difference operators, twisted
  $q$-{$W$} algebras and conformal blocks.
\newblock Preprint
  [\href{https://arxiv.org/abs/1906.00600}{\texttt{arXiv:1906.00600}}], 2019.

\bibitem[BP98]{Bouwknegt:1998}
Peter Bouwknegt and Krzysztof Pilch.
\newblock The deformed {V}irasoro algebra at roots of unity.
\newblock {\em Comm. Math. Phys.}, 196(2):249--288, 1998.
\newblock
  [\href{https://arxiv.org/abs/q-alg/9710026}{\texttt{arXiv:q-alg/9710026}}].

\bibitem[DI97]{Ding1997}
Jintai Ding and Kenji Iohara.
\newblock Drinfeld comultiplication and vertex operators.
\newblock {\em J. Geom. Phys.}, 23(1):1--13, 1997.
\newblock [\href{https://arxiv.org/abs/q-alg/9608003}
  {\texttt{arXiv:q-alg/9608003}}].

\bibitem[FF96]{FF}
B.~Feigin and E.~Frenkel.
\newblock Quantum {${\cal W}$}-algebras and elliptic algebras.
\newblock {\em Comm. Math. Phys.}, 178(3):653--678, 1996.
\newblock [\href{https://arxiv.org/abs/q-alg/9508009}
  {\texttt{arXiv:9508009}}].

\bibitem[FFJ{\etalchar{+}}11]{FFJMM}
B.~Feigin, E.~Feigin, M.~Jimbo, T.~Miwa, and E.~Mukhin.
\newblock Quantum continuous {$\mathfrak{gl}_\infty$}: tensor products of
  {F}ock modules and {${\cal W}_n$}-characters.
\newblock {\em Kyoto J. Math.}, 51(2):365--392, 2011.
\newblock [\href{https://arxiv.org/abs/1002.3113} {\texttt{arXiv:1002.3113}}].

\bibitem[FHS{\etalchar{+}}10]{FHSSY}
B.~Feigin, A.~Hoshino, J.~Shibahara, J.~Shiraishi, and S.~Yanagida.
\newblock Kernel function and quantum algebra.
\newblock {\em RIMS Kokyuroku}, 1689:133--152, 2010.
\newblock [\href{https://arxiv.org/abs/1002.2485v1}{\texttt{arXiv:1002.2485}}].

\bibitem[GN17]{GN15}
E.~Gorsky and A.~Negu\c{t}.
\newblock Infinitesimal change of stable basis.
\newblock {\em Selecta Math. (N.S.)}, 23(3):1909--1930, 2017.
\newblock [\href{https://arxiv.org/abs/1510.07964}
  {\texttt{arXiv:1510.07964}}].

\bibitem[JM95]{JM}
Michio Jimbo and Tetsuji Miwa.
\newblock {\em Algebraic analysis of solvable lattice models}, volume~85 of
  {\em CBMS Regional Conference Series in Mathematics}.
\newblock Published for the Conference Board of the Mathematical Sciences,
  Washington, DC; by the American Mathematical Society, Providence, RI, 1995.

\bibitem[Koy94]{Koyama1993}
Yoshitaka Koyama.
\newblock Staggered polarization of vertex models with {$U_q(\widehat{{\rm
  sl}(n)})$}-symmetry.
\newblock {\em Comm. Math. Phys.}, 164(2):277--291, 1994.
\newblock [\href{https://arxiv.org/abs/hep-th/9307197}
  {\texttt{arXiv:hep-th/9307197 }}].

\bibitem[MN18]{Maillet2018}
J.~M. Maillet and G.~Niccoli.
\newblock On quantum separation of variables.
\newblock {\em J. Math. Phys.}, 59(9):091417, 47, 2018.
\newblock [\href{https://arxiv.org/abs/1807.11572}
  {\texttt{arXiv:1807.11572}}].

\bibitem[Neg18]{N16}
A.~Negu\c{t}.
\newblock The {$q$}-{AGT}-{W} relations via shuffle algebras.
\newblock {\em Comm. Math. Phys.}, 358(1):101--170, 2018.
\newblock [\href{http://arxiv.org/abs/1608.08613} {\texttt{arXiv:1608.08613}}].

\bibitem[Shi04]{Sh}
J.~Shiraishi.
\newblock Free field constructions for the elliptic algebra {${\cal
  A}_{q,p}(\widehat{\rm sl}_2)$} and {B}axter's eight-vertex model.
\newblock {\em Internat. J. Modern Phys. A}, 19(May, suppl.):363--380, 2004.
\newblock
  [\href{https://arxiv.org/abs/math/0302097v2}{\texttt{arXiv:0302097}}].

\bibitem[SKAO96]{SKAO}
J.~Shiraishi, H.~Kubo, H.~Awata, and S.~Odake.
\newblock A quantum deformation of the {V}irasoro algebra and the {M}acdonald
  symmetric functions.
\newblock {\em Lett. Math. Phys.}, 38(1):33--51, 1996.
\newblock [\href{https://arxiv.org/abs/q-alg/9507034}{\texttt{arXiv:9507034}}].

\end{thebibliography}
	
	\noindent \textsc{Landau Institute for Theoretical Physics, Chernogolovka, Russia,\\
		Center for Advanced Studies, Skolkovo Institute of Science and Technology, Moscow, Russia,\\
		National Research University Higher School of Economics, Moscow, Russia,\\
		Institute for Information Transmission Problems, Moscow, Russia,\\
		Independent University of Moscow, Moscow, Russia}
	
	\emph{E-mail}:\,\,\textbf{mbersht@gmail.com}\\

	\noindent\textsc{Center for Advanced Studies, Skolkovo Institute of Science and Technology, Moscow, Russia\\
		National Research University Higher School of Economics, Moscow, Russia}
	
	\emph{E-mail}:\,\,\textbf{roma-gonin@yandex.ru}

\end{document}